\documentclass[ECP,preprint]{ejpecp} 

\usepackage{setspace}
\usepackage{enumitem}
\usepackage{xcolor}
\usepackage{cleveref}
\usepackage[square,numbers]{natbib}
\usepackage{subcaption}

\SHORTTITLE{Community Boundaries in Geometric Graphs}

\TITLE{Estimating Community Boundaries in Geometric Random Graphs
} 

\AUTHORS{%
  Taha~Ameen
    \footnote{University of Illinois, United States of America.
        \EMAIL{tahaa3@illinois.edu}}\orcid{0000-0001-5449-0840}
  \and 
  Neeladri~Maitra
    \footnote{University of Illinois, United States of America.
        \EMAIL{nmaitra@illinois.edu} }
}


\AMSSUBJ{60D05; 05C80} 

\SUBMITTED{August 12, 2026} 
\ACCEPTED{December 13, 2014} 

\VOLUME{0}
\YEAR{2023}
\PAPERNUM{0}
\DOI{10.1214/YY-TN}

\ABSTRACT{
    The unit square \(I=[0,1]^2\) is divided into two rectangles by the vertical line \(x=p\), where \(p\in(0,1)\). Consider \(N\) independent uniformly distributed points on \(I\), which we interpret as a population of individuals, with the line \(x=p\) representing a community boundary that separates the population into two communities. Whether a pair of individuals share a connection depends on their locations in \(I\) and on whether they belong to the same community. Specifically, two individuals in the same community are connected if they are within distance \(R_N\) of each other, while two individuals in different communities are connected if they are within distance \(R_N'\) of each other, giving rise to a geometric variant of the so-called \emph{stochastic block model}. A statistician observes the adjacency matrix of the resulting graph together with the geometric locations of the individuals and is tasked with estimating the boundary location \(p\). Depending on how \(R_N\) and \(R_N'\) scale as \(N\to\infty\), we establish necessary and sufficient conditions for consistent estimation of \(p\). Whenever consistent estimation is possible, we devise an estimator that converges to \(p\) as \(N\to\infty\) and provide explicit bounds on its estimation error.
}

\DeclareMathOperator*{\argmin}{arg\,min}

\newcommand{\1}{\mathbf 1}
\renewcommand{\d}{\mathrm d}
\newcommand{\PP}{\mathbb P}
\newcommand{\EE}{\mathbb E}

\newcommand{\diam}{\operatorname{diam}}
\newcommand{\Var}{\operatorname{Var}}

\newcommand{\cA}{\mathcal A}
\newcommand{\cC}{\mathcal C}

\newcommand{\cD}{\mathcal D}
\newcommand{\R}{\mathbb R}
\newcommand{\unitsquare}{S}

\begin{document}


\section{Introduction}
We study the problem of community detection in a spatial population, where the communities are separated by a geometric boundary. Formally, let \( N \) be a positive integer and define the unit square
\(
    \unitsquare \coloneqq [0,1]^2 \, .
\)
For each \(N\), let
\[
    X_i =(X_i^{(1)},X_i^{(2)}) \, ,  ~~~~ i = 1,\cdots,N,
\]
be independent uniform random points in \(\unitsquare\), which we think of as geographic locations of individuals in a network.  Let \(p \in (0,1)\) be an unknown vertical community boundary.  The two communities are distinguished according to whether an individual with location $(X_i^{(1)},X_i^{(2)})=(x,y)$ is to the left or right of $p$, i.e., whether $x<p$ or $x>p.$ 

Let 
\(
    R_N, R_N' \ge 0.
\)
For 
\(
    1 \le i < j \le N \, ,
\)
define the edge indicator \(A_{ij}\) as follows.  If \(X_i\) and \(X_j\) are on the same side of the boundary, then
\[
    A_{ij} = 1
    ~~~~ \Longleftrightarrow ~~~~
    \|X_i-X_j\| \le R_N \, .
\]
If \(X_i\) and \(X_j\) are on opposite sides of the boundary, then
\[
    A_{ij}=1
    ~~~~ \Longleftrightarrow ~~~~
    \|X_i-X_j\| \le R_N' \, .
\]
Call the constructed graph \( G(N,p,R_N,R'_N) \). The data
\(
    \cD_N=\big((X_i)_{i=1}^N \, , \, (A_{ij})_{1\le i<j\le N} \big),
\)
is observed, including both the vertex locations and the graph adjacency. Figure~\ref{fig:observation} presents an illustration of the observed data. The goal is to estimate \(p\).

\begin{figure}[t]
    \centering
    \begin{subfigure}[t]{0.3\linewidth}
        \centering
        \includegraphics[width=\linewidth]
            {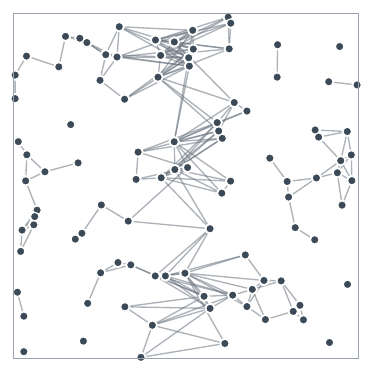}
        \caption{\( R_N = 0.1, R_N' = 0.25 \)}
        \label{fig:observation-1}
    \end{subfigure}
    \hfill
    \begin{subfigure}[t]{0.3\linewidth}
        \centering
        \includegraphics[width=\linewidth]
            {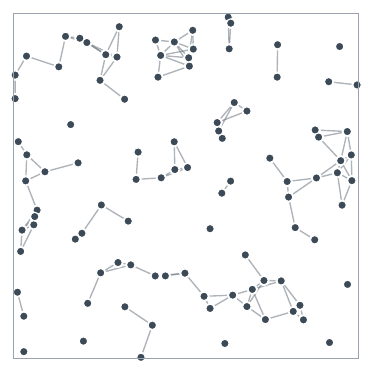}
        \caption{\( R_N = 0.1, R_N' = 0.1 \)}
        \label{fig:observation-2}
    \end{subfigure}
    \hfill
    \begin{subfigure}[t]{0.3\linewidth}
        \centering
        \includegraphics[width=\linewidth]
            {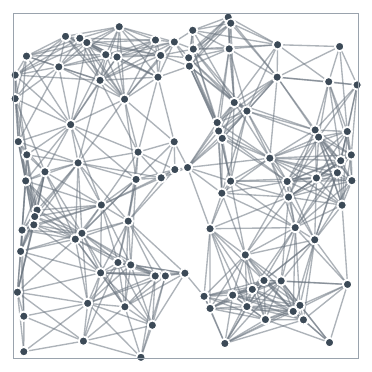}
        \caption{\( R_N = 0.25, R_N' = 0.1 \)}
        \label{fig:observation-3}
    \end{subfigure}
    \caption{The observed information. The graph \( G(N,p,R_N, R'_N)\) for various \( R_N \) and \( R'_N \). The true value of $p$ is $0.5$ in the above figures.}
    \label{fig:observation}
\end{figure}

\begin{remark}[Random geometric graphs]
    When $R_N=R'_N$, $p$ loses its relevance, and the model \( G(N,p,R_N,R'_N) \) reduces to a \emph{random geometric graph} (e.g., see \cite{penrose2003random}) on $[0,1]^2$ with $N$ points and connection radius $R_N$.
\end{remark}

\subsection{Objectives, Estimation Algorithm and Main Result}

Throughout the paper, $p \in (0,1)$ is fixed and does not depend on $N$. 
\begin{definition}[Pointwise consistency]
     An estimator \(\widetilde p_N\) is consistent on \( (0,1) \) if, for every fixed \( p \in (0,1) \) and every \(\varepsilon > 0\),
    \[
        \lim_{N\to\infty} 
        \PP_p \big( |\widetilde p_N-p|>\varepsilon \big) = 0 \, .
    \]
    Here $\PP_p$ emphasizes the dependence of the probability measure on the parameter \(p\).
\end{definition}
Often when it is clear from the context that we are working with a fixed $p$, we drop the dependence of $\PP_p(\cdot)$ on $p$ and simply write $\mathbb{P}(\cdot)$. Throughout, we denote
\[
    s_N \coloneqq \min\{R_N,R_N'\} \, ,
    ~~~~~~
    t_N \coloneqq \max\{R_N,R_N'\} \, .
\]

\paragraph{Estimation Algorithm.}

For \(i<j\), define the horizontal interval
\[
    I_{ij}
    =
    \left(
        \min\{ X_i^{(1)},X_j^{(1)} \} \, ,
        \max\{ X_i^{(1)},X_j^{(1)} \}
    \right) \, .
\]
Also define the set of \emph{informative pairs}
\[
    \cA_N
    =
    \left\{
        (i,j): 1 \le i < j \le N, \  s_N<\|X_i-X_j\|\le t_N
    \right\}.
\]
To justify the name, pairs outside \(\cA_N\) contain no information about \(p\): pairs with distance at most \(s_N\) have the same edge status (they are always present), and pairs with distance greater than \(t_N\) also have the same edge status (they are always absent).
For \((i,j)\in\cA_N\), the edge indicator reveals whether the interval \(I_{ij}\) contains \(p\).  Define
\[
    Y_{ij}
    \coloneqq
    \begin{cases}
        A_{ij},   &  R_N' > R_N \,  , \\
        1-A_{ij}, &  R_N  > R_N' \, .
\end{cases}
\]
Then, for every \((i,j)\in\cA_N\), it holds that
\(
    Y_{ij}
    =
    \1\{p\in I_{ij}\}.
\)
Define the feasible set of boundary locations consistent with the observed data by
\[
    \widehat{\cC}_N
    =
    \left\{
        q\in [0,1] :
        \1\{q\in I_{ij}\}=Y_{ij}
        \text{ for every }(i,j)\in\cA_N
    \right\}.
\]
The true \(p\) belongs to \(\widehat{\cC}_N\), so this set is nonempty almost surely (in case $\mathcal{A}_N$ is empty, by convention $\widehat{\cC}_N=[0,1]$).
We consider the estimator
\[
    \widehat p_N
    \coloneqq
    \frac12
    \left(
        \inf\widehat{\cC}_N \, + \,  \sup\widehat{\cC}_N
    \right) \, .
\]

\paragraph{Main result}
Our main result, presented below, is a necessary and sufficient condition for consistent estimation of $p$, along with an upper bound on the error rate.

\medskip 

\begin{theorem} \label{thm:main-result}
Assume \(0 \le R_N,R_N' \le 1/4\). Let \(\widehat p_N\) be the feasible-set estimator defined above.
\begin{enumerate}[label=\textup{(\alph*)}]
    \item If
    \(
        N^2\big|R_N^3-(R_N')^3\big|\to\infty \, ,
    \)
    then, for every \(p \in (0,1) \) and every fixed \(M\ge1\),
    \[
        \PP_p\left(
        |\widehat p_N-p|
        >
        2M\left[
            \frac1N+
            \frac{1}{N^2|R_N^2-(R_N')^2|}
        \right]
        \right)
        \le
        \frac{C_p}{M}+o(1) \, ,
    \]
    where \(C_p < \infty\) depends only on \(p\). In particular, \(\widehat p_N\) is consistent on \( (0,1) \). 
    \item If
    \(
        N^2\big|R_N^3-(R_N')^3\big|\not\to\infty,
    \)
    then no estimator sequence is consistent on \( (0,1) \).
\end{enumerate}
\end{theorem}

\begin{remark}
    It can be shown that the same estimator is \emph{uniformly} consistent: for every fixed compact set $K$ of $(0,1)$, the conclusion of part (a) holds with
    \[
        \sup_{p \in K} \PP_p \Big( |\widehat p_N - p | > 2M \Big[ \frac1N + \frac{1}{N^2 |R_N^2 - (R'_N)^2 | }\Big] \Big) \leq \frac{C_K}{M} + o(1) \, .
    \]
    This follows from the proof of part (a), where it is shown that the geometric lower bounds depend on $p$ only though its distance from $\{0,1\}$, which is bounded from below on $K$.
\end{remark}    

\subsection{Related literature} The problem of detecting the presence of communities in networks has a rich history \cite{fortunato2010community,fortunato2016community,lancichinetti2009community,leskovec2010empirical,yang2016comparative,magnani2021community,verzelen2015community}. Most of the rigorous mathematical approaches have focused on the \emph{Stochastic Block Model} (SBM) \cite{mossel2012stochastic, lee2019review, decelle2011asymptotic, deshpande2018contextual, funke2019stochastic}, where both the assignment of communities, and, given the assignments, the inter and intra community edge occurrences, are decided via independent coin flips. We refer the reader to the survey by Abbe \cite{abbe2018community} for a comprehensive account of this direction. Our model $G(N,p,R_N,R'_N)$ can be viewed as a geometric variant of the SBM, where the randomness comes only from the vertex locations, while the connections between vertex pairs and the community separation, are both geometric. Certain geometric variants of the SBM have been considered in the literature before, e.g., \cite{gaudio2024exact,avrachenkov2024community,gaudio2025exact,abbe2021community, sankararaman2018community}, where the vertex locations are geometric, however, the community assignments are still independent coin-flips. One motivation to study $G(N,p,R_N,R'_N)$ is to look at models for real-life networks where the community separation is also geometric, e.g., geographic borders separating countries or cities. Our Theorem \ref{thm:main-result} then characterizes when a `border reconstruction' is possible from observed network data and individual geometric locations, under the mathematically simplistic, but not unrealistic assumption that the border is vertical\footnote{For example, the border between the US states of Colorado and Utah.}, as a first step in this rich direction.


\section{Preliminaries: Geometric estimates}

The proofs use several properties of annuli and annular pairs. In this section, we collect some of these properties. 

For \(0\le s<t\), define the annular chord function
\[
    g_{s,t}(h)
    =
    \mu_1\left\{
        y\in\R:
        s^2<h^2+y^2\le t^2
        \right\},
    \qquad h\ge0,
\]
where \(\mu_1\) denotes the one-dimensional Lebesgue measure.  Explicitly, we have
\[
    g_{s,t}(h)
    =
    2\left[
    \sqrt{(t^2-h^2)_+}
    -
    \sqrt{(s^2-h^2)_+}
    \right] 
    \, = \,  
    \begin{cases}
        2 \big(  \sqrt{t^2 - h^2} - \sqrt{ s^2 - h^2} \big)  \, , & 0 \leq h < s \\
        2\sqrt{t^2 - h^2} \, , & s \leq h \leq t \\ 
        0 \, , &  h \geq t 
    \end{cases}
    \, .
\]
Our first technical lemma concerns various integral estimates of this function.

\begin{figure}[t]
    \centering
    \begin{subfigure}[t]{0.45\linewidth}
        \centering
        \includegraphics[width=\linewidth]
            {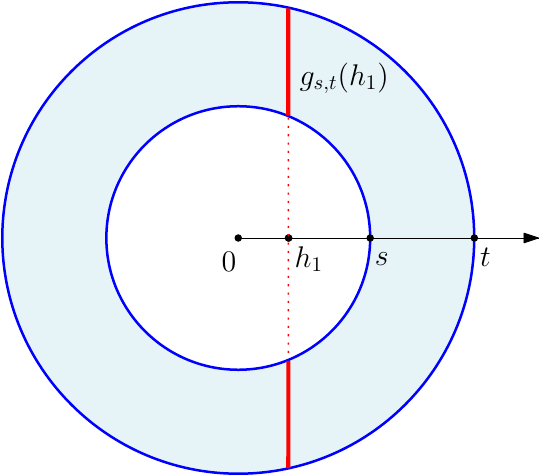}
        \caption{\(h_1\in[0,s]\)}
        \label{fig:definition-of-g-1}
    \end{subfigure}
    \hfill
    \begin{subfigure}[t]{0.45\linewidth}
        \centering
        \includegraphics[width=\linewidth]
            {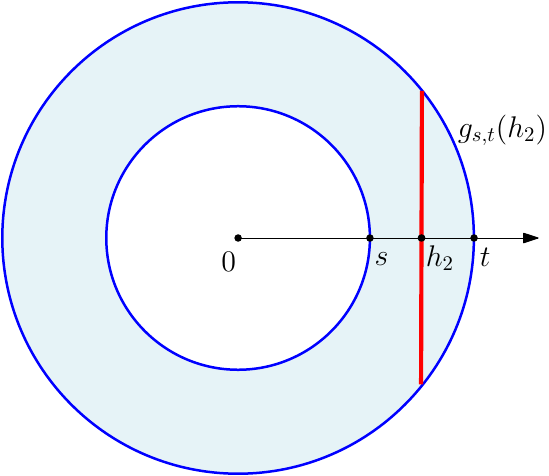}
        \caption{\(h_2\in[s,t]\)}
        \label{fig:definition-of-g-2}
    \end{subfigure}
    \caption{Geometric interpretation of \(g_{s,t}(h)\) when
    \(h\in[0,s]\) in panel~\textup{(a)} and \(h\in[s,t]\) in
    panel~\textup{(b)}. In each panel, \(g_{s,t}(h)\) is the length
    of the vertical chord at horizontal displacement \(h\) lying
    inside the annulus with inner radius \(s\) and outer radius \(t\).}
    \label{fig:definition-of-g}
\end{figure}

\medskip 

\begin{lemma}[Annular chord estimates] \label{lem:chord-estimates}
Let
\(
    a = t^2-s^2
\)
and
\(
    b=t^3-s^3 \, .
\)
Suppose that
\(
    0\le s<t\le 1/4 \, .
\)
The following holds.

\begin{enumerate}[label=\textup{(\roman*)}]
\item
\(  \displaystyle
    \begin{aligned}[t]
    \int_0^\infty g_{s,t}(h)\, \d h
    =
    \frac{\pi}{2} \, a  \, .
    \end{aligned}
\)

\item
\( \displaystyle
    \begin{aligned}[t]
    \int_0^\infty h\,g_{s,t}(h)\, \d h
    =
    \frac{2}{3} \, b \, .
    \end{aligned}
\)

\item There exists a universal constant \(c_0>0\) such that, for every
    \(0\le r\le c_0t\),
    \begin{equation} \label{eq-geom-(iv)-i}
        \frac{\pi}{4} a
        \, \le \,
        \int_r^\infty g_{s,t}(h)\,\d h
        \, \le \, 
        \frac{\pi}{2}a \, .
    \end{equation}
    In addition, for every $0 \leq r \leq t/4$, 
    \begin{equation} \label{eq-geom-(iv)-ii}
        \int_0^r g_{s,t}(h) \, \d h \leq \frac{8 \,  r a}{t} \, .
    \end{equation}
\end{enumerate}
\end{lemma}

\begin{proof}~ (i) The annular chord function $g_{s,t}(h)$ is identically zero when $h > t$. Integrating it over $[0,t]$ gives us the area of the half-annulus
\(
    \{(h,y): h\ge0, \ s^2<h^2+y^2 \le t^2\} \, .
\)
Indeed,
\(
    \int_0^\infty g_{s,t}(h)\, \d h
    =
    \frac12\pi \, (t^2-s^2)
    =
    \frac{\pi}{2} \, a \, .
\)

\medskip 

\noindent (ii) In polar coordinates, one has
\[
    \int_0^\infty h\,g_{s,t}(h)\, \d h
    =
    \int_s^t
    \int_{-\pi/2}^{\pi/2}
    (r\cos\theta) \, r\, \d \theta\, \d r
    =
    2\int_s^t r^2\, \d r
    =
    \frac23 \, (t^3-s^3) \, .
\]

\medskip 

\noindent (iii) Recall from (i) that
\(
    \int_0^\infty g_{s,t}(h)\,\d h
    =
    \frac{\pi}{2}a.
\)
It suffices then to show that a sufficiently small initial segment \([0,c_0t]\) contains at most half of the total mass. We first show that for \(c_0\le 1/4\),
\begin{align} \label{eq-preliminary-int}
    \int_0^{c_0t} g_{s,t}(h)\, \d h
    \, \le \,  8 c_0 a \, .
\end{align}
Consider the following two cases.
\begin{itemize}
    \item Case 1: \( s \le 2 c_0 t\). Here we have
    \(
        a=t^2-s^2 
        \, \ge \,  
        (1-4c_0^2) \, t^2 
        \, \ge \, 
        \frac34 t^2 \, ,
    \)
    and since \(g_{s,t}(h)\le 2t\),
    \[
        \int_0^{c_0t}g_{s,t}(h)\,\d h
        \le 2c_0t^2
        \le 8c_0a \, .
    \]
    \item Case 2: \( s > 2c_0 t \). For \(0\le h\le c_0t\) we have \(h<s\), so
    \(
        g_{s,t}(h)
        \, = \, 
        \frac{2a}{
            \sqrt{t^2-h^2}+\sqrt{s^2-h^2}
        }.
    \)
    Also \(h\le c_0t\le t/4\), hence
    \(
        \sqrt{t^2-h^2}\ge \frac34t.
    \)
    Therefore, we have that
    \(
        g_{s,t}(h)\le \frac{8a}{3t},
    \)
    and so
    \[
        \int_0^{c_0t}g_{s,t}(h)\,\d h
        \, \le \,  \frac83 c_0a
        \, \le \,  8c_0a \, .
    \]
\end{itemize}
Now choose \(c_0>0\) so small that
\(
    c_0\le \frac14
\)
and
\(
    8c_0\le \frac{\pi}{4} \, .
\)
Then
\[
    \int_0^{c_0t}g_{s,t}(h)\,\d h
    \, \le \,
    \frac{\pi}{4}a
    \, = \,
    \frac12\int_0^\infty g_{s,t}(h)\,\d h \, . 
\]
Consequently, for every \(0\le r\le c_0t\),
\[
    \int_r^\infty g_{s,t}(h)\,\d h
    \, \ge \,
    \int_{c_0t}^\infty g_{s,t}(h)\,\d h
    \, \ge \,
    \frac12\int_0^\infty g_{s,t}(h)\,\d h
    \, = \,
    \frac{\pi}{4}a \, .
\]
The reverse inequality follows immediately from (i):
\[
    \int_r^\infty g_{s,t}(h)\,\d h
    \, \le \,
    \int_0^\infty g_{s,t}(h)\,\d h
    \, = \,
    \frac{\pi}{2}a \, .
\]
This concludes the proof of~\eqref{eq-geom-(iv)-i}.

Finally, if $r \leq t/4$,~\eqref{eq-geom-(iv)-ii} follows by setting $c_0 = r/t$ in~\eqref{eq-preliminary-int}
\end{proof}

The next lemma concerns the probability that two independent points sampled uniformly at random are on either side of a boundary, while satisfying a distance constraint. Let \(U=(U^{(1)},U^{(2)}),V=(V^{(1)},V^{(2)})\) be independent uniform points in \([0,1]^2\).  Define their horizontal interval
\[
    I(U,V)
    =
    \left(
        \min\{U^{(1)}, \, V^{(1)}\}  , \;
        \max\{U^{(1)}, \, V^{(1)}\}
    \right) \, .
\]

\begin{lemma}[Crossing probability] \label{lem:crossing-probability}
Fix \(z\in (0,1)\).  For \(0\le s<t\le 1/4\),
\begin{equation} \label{eq:cross-i}
    \PP\Big(
    s< \|U-V\| \le t \, , \
    z\in I(U,V)
    \Big)
    = \Theta( t^3-s^3 ) \, .
\end{equation}
Moreover, if \(z,w\in (0,1) \)  are fixed, and \(|z-w|\ge2t\), then
\begin{equation} \label{eq:cross-ii}
    \PP\Big(
        s<\|U-V\|\le t,\;
        \1\{z\in I(U,V)\}\ne\1\{w\in I(U,V)\}
    \Big)
    = \Theta( t^3-s^3 ) \, .
\end{equation}
\end{lemma}

\begin{proof}
Denote the event
\[
    C_z
    \coloneqq
    \left\{
        s<\|U-V\|\le t,\ z\in I(U,V)
    \right\}.
\]
First consider the orientation \(U^{(1)}<z<V^{(1)}\), and define their horizontal separation and signed vertical displacement
\[
    h=V^{(1)}-U^{(1)},
    \qquad
    y=V^{(2)}-U^{(2)}.
\]
Also set \(\rho_z=\min\{z,1-z\}\). For fixed \(h\in(0,t)\), for the event $C_z$ to hold, the set of
admissible values of \(U^{(1)}\) has length
\(
    \ell_z(h)=\min\{h,\rho_z\}.
\) 
Indeed, this follows directly by intersecting \((z-h,z)\) with
\([0,1-h]\). Moreover, defining $\kappa_z\coloneqq\min\{1,4\rho_z\}>0$,
\[
    \kappa_z h\le \ell_z(h)\le h,
\]
because \(h\le t\le1/4\).
Now, the random variable \(V^{(2)}-U^{(2)}\) has density \(1-|x|\) for
\(x\in[-1,1]\). Accounting for the two possible horizontal orientations give
the identity
\[
\PP(C_z)
=
2\int_0^t \ell_z(h)
\int_{\mathbb R}
\1\{s^2<h^2+y^2\le t^2\}(1-|y|)\,\d y\,\d h.
\]
On the support of the inner integrand, \(|y|\le t\). Consequently,
\[
(1-t)g_{s,t}(h)
\le
\int_{\mathbb R}
\1\{s^2<h^2+y^2\le t^2\}(1-|y|)\,\d y
\le
g_{s,t}(h).
\]
Since \(t\le1/4\), Lemma~\ref{lem:chord-estimates}(ii)
now yields
\[
    \kappa_z(t^3-s^3)
    \le
    \PP(C_z)
    \le
    \frac43(t^3-s^3).
\]
This proves~\eqref{eq:cross-i} with constants depending only on the fixed point \(z\).

We now prove~\eqref{eq:cross-ii}. Let \(C_z\) and \(C_w\) denote the corresponding
crossing events. If both events occur, then \(z,w\in I(U,V)\), and hence
\[
    |z-w|
    <
    |U^{(1)}-V^{(1)}|
    \le
    \|U-V\|
    \le t,
\]
contradicting \(|z-w|\ge2t\). Thus \(C_z\) and \(C_w\) are disjoint,
and the event that the two indicators differ is precisely \(C_z\cup C_w\). Applying~\eqref{eq:cross-i} at \(z\) and \(w\) completes
the proof.
\end{proof}

Our final lemma of this section concerns a \emph{local endpoint probability}: it studies the probability that two independent uniform points $U$ and $V$ are (i) appropriately spaced, i.e., their Euclidean distance is between $s$ and $t$, (ii) they are on opposite sides of the vertical line at a fixed $z$, and (iii) one of the points is $\delta$-close in horizontal distance to $z$. We call an informative interval \emph{positive} when its label is \(Y_{ij}=1\), equivalently when it contains the true boundary. To localize \(p\) to a band of width \(2\delta\), it is enough to find positive informative intervals with endpoints within \(\delta\) of \(p\) on both sides. The following estimate controls the probability of these events.

\medskip 

\begin{lemma}[Local endpoint probability] \label{lem:local-endpoint}
    Fix \(z\in (0,1)\), and assume \(0\le s<t\le 1/4\).  There exist constants \(c_0>0\), and \(0<c_1<c_2<\infty\) depending only on $z$ such that, whenever
    \(
        0<\delta\le c_0t \, ,
    \)
    one has
    \begin{equation} \label{eq-local-ep}
        c_1\delta(t^2-s^2)
        \le
        \PP\left(
        U^{(1)}<z<V^{(1)}<z+\delta,\;
        s<\|U-V\|\le t
        \right)
        \le
        c_2 \, \delta(t^2-s^2) \, ,
    \end{equation}
    uniformly in \(s,t,\delta\).  The analogous estimate with
    \(
        z-\delta<U^{(1)} < z < V^{(1)}
    \)
    also holds.
\end{lemma}

\begin{proof}
We prove the right-endpoint estimate. Let \(P_+\) denote the probability in~\eqref{eq-local-ep} and
\[
    a \coloneqq t^2-s^2,
    ~~~~
    \rho_z = \min\{z,1-z\},
    ~~~~
    \eta_z=\min\Big\{4\rho_z,\frac12\Big\}.
\]
For \(0\le h\le t/2\), the explicit formula for \(g_{s,t}\) gives
\(
    g_{s,t}(h)\ge {a}/{t}.
\)
Indeed, if \(h<s\), then
\[
    g_{s,t}(h)
    =
    \frac{2a}{
        \sqrt{t^2-h^2}+\sqrt{s^2-h^2}
    }
    \ge\frac{a}{t},
\]
whereas if \(h\ge s\), then
\[
    g_{s,t}(h)
    =
    2\sqrt{t^2-h^2}
    \ge\sqrt{3}\,t
    \ge\frac{a}{t} \, .
\]
Consequently, since \(t\le1/4\) and \(\rho_z\le z\),
\[
    \int_0^z g_{s,t}(h)\,\d h
    \ge
    \frac{a}{t}\min\Big\{z,\frac t2\Big\}
    \ge
    \eta_z a \, .
\]
Next, choose
\(
    c_0(z)= {\eta_z}/{16}.
\)
If \(0<\delta\le c_0(z)t\), then \(\delta\le\rho_z\) and
\(\delta\le t/4\). Write \(V=(z+r,v)\), where \(0<r<\delta\).
For \(v\in[t,1-t]\), the annulus around \(V\) is not truncated by
the top or bottom sides of the square, and the area of admissible
values of \(U\) is
\(
    \int_r^{z+r}g_{s,t}(h)\,\d h.
\)
For every \(0\le r\le\delta\), Lemma~\ref{lem:chord-estimates}(iii)
therefore gives
\[
\begin{aligned}
    \int_r^{z+r}g_{s,t}(h)\,\d h
    &\ge
    \int_r^z g_{s,t}(h)\,\d h 
    \ge
    \eta_z a-\int_0^r g_{s,t}(h)\,\d h 
    \ge
    \left(\eta_z-8\frac{r}{t}\right)a
    \ge
    \frac{\eta_z}{2}a.
\end{aligned}
\]
Since \(1-2t\ge1/2\), integrating over
\(0<r<\delta\) and \(t<v<1-t\) yields
\[
    P_+
    \ge
    (1-2t)\delta\frac{\eta_z}{2}a
    \ge
    \frac{\eta_z}{4}\delta a.
\]
Conversely, for each \(V\) with \(z<V^{(1)}<z+\delta\), the admissible values of \(U\) form a subset of a half-annulus of area \(\pi a/2\). Hence
\(
    P_+
    \le
    \frac{\pi}{2}\delta a.
\)
Thus one may take
\(
    c_1(z)={\eta_z}/{4},
\)
and
\(
    c_2(z)={\pi}/{2}.
\)
The left-endpoint estimate follows by horizontal reflection and exchangeability of \(U\) and \(V\).
\end{proof}

\section{A technical lemma on endpoint localization}

The next lemma is the technical ingredient behind the rate.  To localize \(p\), we want positive informative intervals whose endpoints lie very close to \(p\).  For a candidate location \(z\), the random variable \(K_N^+(z,\delta)\) counts informative intervals that cross \(z\) and have their right endpoint in \((z,z+\delta)\).  Similarly, \(K_N^-(z,\delta)\) counts informative intervals that cross \(z\) and have their left endpoint in \((z-\delta,z)\).  If both counts are nonzero with \(z=p\), then the intervals squeeze the feasible set from both right and left.

Recall that \( s_N = \min\{R_N, R_N'\} \) and \( t_N = \max\{R_N, R_N'\}\). For \(z\in(0,1)\) and \(\delta>0\), define
\[
    K_N^+(z,\delta)
    =
    \sum_{i\ne j} E^+_{ij} , 
    ~~~~ \mbox{where }
    E^+_{ij} \coloneqq
    \1\Big\{
    X_i^{(1)}<z<X_j^{(1)}<z+\delta, \ 
    s_N<\|X_i-X_j\|\le t_N
    \Big\} \, ,
\]
and
\[
    K_N^-(z,\delta)
    =
    \sum_{i\ne j} E^-_{ij} , 
    ~~~~ \mbox{where }
    E^-_{ij} \coloneqq \1\Big\{
    z-\delta<X_i^{(1)}<z<X_j^{(1)}, \
    s_N<\|X_i-X_j\|\le t_N
    \Big\} \, .
\]

\begin{lemma}[Endpoint localization] \label{lem:localization}
Assume that \(0\le s_N<t_N\le1/4\) for all sufficiently large \(N\), and fix \(z\in(0,1)\). Let \(a_N = t_N^2 - s_N^2 \), \(b_N = t_N^3 - s_N^3 \) and 
\(
    \Delta_N
    =
    \frac1N+\frac{1}{N^2a_N}.
\) 
If
\(
    N^2 b_N\to\infty \, ,
\)
then, for every fixed \(M\ge1\),
\[
    \PP\left(
        K_N^+(z,M\Delta_N)=0
    \right)
    \, \le \,
    \frac{C_z}{M}+o(1) \, ,
    ~~~~
    \mbox{and}
    ~~~~
    \PP\left(
        K_N^-(z,M\Delta_N)=0
    \right)
    \, \le \, 
    \frac{C_z}{M}+o(1) \, ,
\]
where \(C_z<\infty\) depends only on the fixed point \(z\).
\end{lemma}

\begin{proof}
We prove the statement for \(K_N^+\). Throughout, denote
\(
    s=s_N, \; 
    t=t_N, \;
    a=t^2-s^2, \;
    b=t^3-s^3,
\)
and set
\[
    \delta=M\Delta_N
    =
    M\left(
        \frac1N+\frac{1}{N^2a}
    \right).
\]
Since
\(
    b/a
    =
    {(t^2+ts+s^2)}/{(t+s)}
    \le 3t,
\)
the assumption \(N^2b\to\infty\) implies \(N^2ta\to\infty\).
Moreover, \(b\le t^3\), so \(N^2t^3\to\infty\). Hence
\(Nt^{3/2}\to\infty\), and since \(t\le1/4<1\), it follows that
\(Nt\ge Nt^{3/2}\to\infty\).
Therefore, for each fixed \(M\),
\[
    \frac{\delta}{t}
    =
    M\left(
        \frac{1}{Nt}
        +
        \frac{1}{N^2at}
    \right)
    \longrightarrow0.
\]
Thus, for all sufficiently large \(N\), we have
\(\delta\le c_0(z)t\) and may apply
Lemma~\ref{lem:local-endpoint}.

Let
\(
    q \coloneqq \EE \big[ E^+_{12} \big] .
\)
The lower bound in Lemma~\ref{lem:local-endpoint} gives
\(
    q\ge c_z\delta a
\)
for some \(c_z>0\). Hence, writing
\[
    \mu=\EE K_N^+(z,\delta)=N(N-1)q,
\]
we have
\(
    \mu\ge c_zN^2\delta a
\)
for all sufficiently large \(N\).

We next bound the variance. Indicators involving disjoint vertex sets are independent. 
If two distinct indicators share a vertex that one event requires to
lie to the left of \(z\) and the other requires to lie to the right of
\(z\), then the two events are disjoint, so the indicators cannot both equal one and their covariance is non-positive.
It remains to consider indicators sharing a vertex in the same role. For example, define
\[
    q_L(x) \coloneqq \EE\left[E^+_{12}\mid X_1=x\right].
\]
Conditional on \(X_1\), the indicators \(E^+_{12}\) and \(E^+_{13}\) are independent. Moreover, for every fixed \(x\), an admissible right endpoint must lie in a subset of the annulus centered at \(x\), whose area is \(\pi a\). Thus \(q_L(x)\le\pi a\), and therefore
\[
    \EE\left[E^+_{12}E^+_{13}\right]
    =
    \EE\left[q_L(X_1)^2\right] 
    \le
    \pi a\,\EE \left[ q_L(X_1) \right]
    =
    \pi a q \, .
\]
The same argument applies to two indicators sharing their right endpoint. Since there are at most \(N(N-1)(N-2)\) covariance terms of each type,
\[
    \Var\left(K_N^+(z,\delta)\right)
    \le
    \mu+2\pi N(N-1)(N-2)aq 
    =
    \mu\left(1+2\pi(N-2)a\right) \, .
\]
Consequently,
\[
    \frac{\Var(K_N^+(z,\delta))}{\mu^2}
    \le
    C_z\left(
        \frac{1}{N^2\delta a}
        +
        \frac{1}{N\delta}
    \right).
\]
Finally, note that
\(
    N^2\delta a=M(Na+1)\ge M,
\)
and
\(
    N\delta
    =
    M\left(1+\frac{1}{Na}\right)
    \ge M.
\)
Therefore, Chebyshev's inequality gives
\[
    \PP\left(K_N^+(z,\delta)=0\right)
    \le
    {
    \PP\left( |K_N^+(z,\delta)-\mu|>\mu/2 \right) 
    }
    \le 
    \frac{4\Var(K_N^+(z,\delta))}{\mu^2}
    \le
    \frac{C_z}{M},
\]
after enlarging \(C_z\) if necessary.

Taking
\(\delta=M\Delta_N\) proves the claimed bound. The proof for \(K_N^-\)
is the same, using the left-endpoint estimate in
Lemma~\ref{lem:local-endpoint} and enlarging \(C_z\) if necessary.
\end{proof}

\section{Proof of~\texorpdfstring{\Cref{thm:main-result}}{}}

\noindent (a) Recall that \( b_N = t_N^3 - s_N^3 \) where \( s_N = \min\{ R_N, R_N' \} \) and \( t_N = \max \{ R_N, R_N' \} \). Assume
\(
    N^2b_N\to\infty \, .
\)
Fix \(p\in(0,1)\). For every informative pair \((i,j)\in\cA_N\), we have
\(
    Y_{ij}
    =
    \1\{p\in I_{ij}\}.
\)
Thus \(p\in\widehat{\cC}_N\).
Define the largest left-endpoint among the informative intervals
\[
    L_N
    \coloneqq
    \max\left\{
        \min\{X_i^{(1)},X_j^{(1)}\}:
        (i,j)\in\cA_N,\;
        p\in I_{ij}
    \right\},
\]
and the smallest right-endpoint among the informative intervals
\[
    U_N
    \coloneqq
    \min\left\{
        \max\{X_i^{(1)},X_j^{(1)}\}:
        (i,j)\in\cA_N,\;
        p\in I_{ij}
    \right\}.
\]
If there is no informative interval containing \(p\), set \(L_N=0\) and \(U_N=1\).  On the event that at least one such interval exists, every \(q\in\widehat{\cC}_N\) must belong to every informative interval containing \(p\).  Hence
\(
    \widehat{\cC}_N\subseteq [L_N,U_N] \, .
\)
Therefore
\[
    \diam(\widehat{\cC}_N)
    \le
    (U_N-p)+(p-L_N) \, .
\]
Let
\(
    \Delta_N
    =
    \frac1N+\frac{1}{N^2a_N} \, .
\)
Fix \(M\ge1\), and set \(\delta=M\Delta_N\).  By Lemma~\ref{lem:localization} applied at \(p\), we have
\begin{align*}
    \PP_p(U_N-p>\delta)
    & \le
    \PP_p\left(K_N^+(p,\delta)=0\right)
    \le
    \frac{C_p}{M}+o(1) \, ,
    \\
    \PP_p(p-L_N>\delta)
    & \le
    \PP_p\left(K_N^-(p,\delta)=0\right)
    \le
    \frac{C_p}{M}+o(1) \, .
\end{align*}
Thus
\[
    \PP_p\Big(
        \diam(\widehat{\cC}_N)>2M\Delta_N
    \Big)
    \le
    \frac{C_p}{M}+o(1) \, .
\]
Since \(p\in\widehat{\cC}_N\), it follows that
\(
    |\widehat p_N-p|
    \le
    \frac12\diam(\widehat{\cC}_N) \, .
\)
Therefore
\[
    \PP_p\left(
        |\widehat p_N-p|> M\Delta_N
    \right)
    \le
    \frac{C_p}{M}+o(1) \, .
\]
To justify the last assertion, note that \(b_N/a_N\le3t_N\le3\), so \(N^2b_N\to\infty\) implies \(N^2a_N\to\infty\) and hence \(\Delta_N\to0\). For any fixed \(\varepsilon>0\) and \(M\ge1\), eventually \(M\Delta_N<\varepsilon\), and so
\[
    \limsup_{N\to\infty}
    \PP_p\bigl(|\widehat p_N-p|>\varepsilon\bigr)
    \le \frac{C_p}{M}.
\]
Letting \(M\to\infty\) proves consistency at the fixed point \(p\). Since the same estimator is used for every \(p\in(0,1)\), this completes part (a).

\bigskip 
\noindent (b) (Converse result) First suppose that
\(
    \Lambda
    \coloneqq
    \sup_N N^2b_N
    <\infty.
\)
We show that no single estimator can be consistent at every fixed
point of \((0,1)\). To this end, choose
\(
    p_0= 1/4,
\)
and
\(
    q_0= 3/4,
\)
so that
\(
    d \coloneqq |p_0-q_0|= 1/2.
\)
Since \(t_N\le1/4\), we have \(d\ge2t_N\).
For \(1\le i<j\le N\), let
\[
    B_{ij}
    =
    \Big\{
        s_N<\|X_i-X_j\|\le t_N,\;
        \1\{p_0\in I_{ij}\}\ne\1\{q_0\in I_{ij}\}
    \Big\},
\]
and define
\(
    S_N
    \coloneqq
    \sum_{1\le i<j\le N}\1_{B_{ij}}.
\)
Thus \(S_N\) counts the informative pairs that distinguish \(p_0\) from \(q_0\). By Lemma~\ref{lem:crossing-probability}, there exists \(C<\infty\) such that
\(
    \PP(B_{ij})\le C \, b_N
\)
for every \(i<j\).

We use the following quantitative form of the general Lov\'asz local
lemma~\cite[Lemma~5.1.1]{alon2016probabilistic}: If \(m\) events have a dependency graph of maximum degree \(D\), and
\[
    \max_\alpha\PP(A_\alpha)
    \le x(1-x)^D
\]
for some \(x\in(0,1)\), then
\(
    \PP \Big(\bigcap_\alpha A_\alpha^c \Big)
    \ge (1-x)^m.
\)
In the present setting, the vertices of the dependency graph are the
unordered vertex-index pairs
\(
    \{i,j\}\in\binom{[N]}2,
\)
corresponding to the events \(B_{ij}\); two such pairs are adjacent whenever they intersect.  This is a dependency graph because events indexed by disjoint
pairs depend on disjoint collections of vertex locations. It has
\(
    m=m_N=\binom N2
\)
and
\(
    D=D_N=2(N-2).
\)

Set
\(
    x_N=2Cb_N.
\)
Since \(b_N\le\Lambda/N^2\), it follows that
\(
    D_N x_N
    \le
    {4C\Lambda}/{N}
    \to 0.
\)
Consequently, for all sufficiently large \(N\), \(x_N\le1/2\) and
\(
    (1-x_N)^{D_N}
    \ge
    1-D_Nx_N
    \ge
    1/2.
\)
It follows that
\[
    \PP(B_{ij})
    \le
    Cb_N
    =
    \frac{x_N}{2}
    \le
    x_N(1-x_N)^{D_N}.
\]
The local lemma therefore gives
\[
    \PP(S_N=0)
    \ge
    (1-x_N)^{m_N} 
    \ge
    \exp(-2m_Nx_N) 
    \ge
    \exp(-2C\Lambda)
    >0,
\]
where we used \(\log(1-x)\ge-2x\) for \(0\le x\le1/2\).

Now construct the observations under \(p_0\) and \(q_0\) on the same probability space using the same vertex locations, and denote them by \(\cD_N^{(0)}\) and \(\cD_N^{(1)}\). On \(\{S_N=0\}\), every informative pair has the same crossing status under \(p_0\) and \(q_0\), while every non-informative pair has the same edge status under every boundary.
Hence
\(
    \cD_N^{(0)}=\cD_N^{(1)}
\)
on
\(
    \{S_N=0\}.
\)
Let \(\widetilde p_N\) be any estimator, and denote
\[
    \widetilde p_N^{(r)}
    \coloneqq
    \widetilde p_N(\cD_N^{(r)}),
    \qquad r\in\{0,1\}.
\]
Set
\(
    \varepsilon_0
    =
    {|p_0-q_0|}/{4}.
\)
On \(\{S_N=0\}\), the two estimator outputs are equal, and a single
number cannot be within \(\varepsilon_0\) of both \(p_0\) and \(q_0\).
Therefore
\[
\begin{aligned}
    &\1\left\{
        |\widetilde p_N^{(0)}-p_0|>\varepsilon_0
    \right\}
    +
    \1\left\{
        |\widetilde p_N^{(1)}-q_0|>\varepsilon_0
    \right\}
    \ge
    \1\{S_N=0\}.
\end{aligned}
\]
Taking expectations and using the marginal laws yields,
for all sufficiently large \(N\),
\[
    \PP_{p_0}^{(N)}
        \left(|\widetilde p_N-p_0|>\varepsilon_0\right)
    +
    \PP_{q_0}^{(N)}
        \left(|\widetilde p_N-q_0|>\varepsilon_0\right) 
    \ge
    \PP(S_N=0)
    \ge
    \exp(-2C\Lambda).
\]
Thus no estimator can be consistent at both \(p_0\) and \(q_0\), and
hence none can be consistent at every fixed \(p\in(0,1)\), when
\(\sup_NN^2b_N<\infty\).

Finally, if \(N^2b_N\not\to\infty\), there exist a constant
\(\Lambda<\infty\) and a subsequence \(N_k\to\infty\) such that
\(
    N_k^2b_{N_k}\le\Lambda
\)
for every \(k\). Applying the preceding argument along this subsequence shows that the two error probabilities above cannot both tend to zero. An estimator consistent at every fixed \(p\in(0,1)\) along the full sequence would remain consistent at \(p_0\) and \(q_0\) along every subsequence, giving a contradiction. This proves part (b).

\section{Consequences and outlook}

The main theorem is governed by two geometric scales: the crossing mass
\(
    b_N=t_N^3-s_N^3
\)
and the annular mass
\(
    a_N=t_N^2-s_N^2.
\)
We first record what these two scales give in several representative regimes. Here, we call the problem \emph{estimable} if there exists a single estimator sequence that is consistent at every fixed \(p\in(0,1)\).  For a positive sequence \(\rho_N\), we say that \(\widehat p_N\) \emph{achieves rate} \(\rho_N\) if, for every fixed \(p\in(0,1)\) and every \(\varepsilon>0\), there exists \(C_{p,\varepsilon}<\infty\) such that
\[
    \PP_p\left(
        |\widehat p_N-p|
        \le C_{p,\varepsilon} \, \rho_N
    \right)
    \ge 1-\varepsilon
\]
for all sufficiently large \(N\).

\begin{corollary}[Special cases]
Assume the conditions of Theorem~\ref{thm:main-result}.
\begin{enumerate}[label=\textup{(\alph*)}]
\item If \(R_N'=0\), then
\(
    a_N=R_N^2
\)
and
\(
    b_N=R_N^3.
\)
Thus the problem is estimable if and only if
\(
    N^2R_N^3\to\infty.
\)
When this holds, \(\widehat p_N\) achieves rate
\(
    \rho_N
    =
    \frac1N+\frac{1}{N^2R_N^2}.
\)

\item Suppose \(R_N\) and \(R_N'\) are close.  Let
\(
    r_N=\max\{R_N,R_N'\}
\)
and
\(
    \delta_N=|R_N-R_N'|.
\)
If \(\delta_N=o(r_N)\), then
\(
    b_N=(3+o(1))r_N^2\delta_N,
\)
and
\(
    a_N=(2+o(1))r_N\delta_N.
\)
Consequently, the problem is estimable if and only if
\(
    N^2r_N^2\delta_N\to\infty.
\)
When this holds, \(\widehat p_N\) achieves rate
\(
    \rho_N
    =
    \frac1N+\frac{1}{N^2r_N\delta_N}.
\)

\item If \(R_N\equiv R\) and \(R_N'\equiv R'\) for two unequal
nonnegative constants satisfying
\(
    \max\{R,R'\}\le1/4,
\)
then \(a_N\) and \(b_N\) are positive constants.  The problem is estimable, and \(\widehat p_N\) achieves rate
\(
    \rho_N=1/N.
\)
\end{enumerate}
\end{corollary}

\noindent We conclude with several directions for future work.

\paragraph{The upper bound $1/4$.} Note that Theorem \ref{thm:main-result} assumes $\max\{R_N,R'_N\}\leq 1/4$. This technical assumption is used in deriving the estimate \eqref{eq-geom-(iv)-ii}; specifically, it allows us to ignore certain boundary effects. It is interesting to think about what happens beyond $1/4$. Obviously, when $R_N=R'_N=\sqrt{2}$, estimation is impossible as in that case $G(N,p,R_N,R'_N)$ is a complete graph. Does this mean there exists a threshold $s^*\in (1/4,\sqrt{2})$ such that estimation is possible only if $\max\{R_N,R_N'\}<s^*$? We leave this as an open problem.

\paragraph{Unknown orientation and higher dimensions.}
In the planar setting, one could replace the vertical boundary by an unknown affine line
\[
    H_{\theta,p}
    =
    \left\{
        x\in[0,1]^2:
        \langle\theta,x\rangle=p
    \right\},
    \qquad
    \|\theta\|_2=1.
\]
The parameters \((\theta,p)\) and \((-\theta,-p)\) describe the same unlabeled partition.  A natural metric on the resulting equivalence classes
is
\[
    d_\pm\bigl((\theta,p),(\varphi,q)\bigr)
    =
    \min_{\sigma\in\{-1,1\}}
    \left\{
        \|\theta-\sigma\varphi\|_2
        +
        |p-\sigma q|
    \right\}.
\]
Let \(\Theta\) be a compact class of normalized lines that cut the square into two parts of positive Lebesgue measure.  The analog of the present feasible set is
\[
    \widehat{\cC}_N
    =
    \left\{
        (\theta,p)\in\Theta:
        \1\{H_{\theta,p}\text{ separates }X_i\text{ and }X_j\}=Y_{ij}
        \text{ for every }(i,j)\in\cA_N
    \right\}.
\]
The midpoint estimator can then be replaced by a minimax geometric center,
for example
\[
    (\widehat\theta_N,\widehat p_N)
    \in
    \argmin_{(\theta,p)\in\Theta}
    \sup_{(\varphi,q)\in\widehat{\cC}_N}
    d_\pm\bigl((\theta,p),(\varphi,q)\bigr).
\]
The same construction should extend to hyperplanes in \([0,1]^d\); there the local annular and crossing masses scale as \(t_N^d-s_N^d\) and
\(t_N^{d+1}-s_N^{d+1}\), respectively.

\paragraph{Curved or nonparametric boundaries.}
One could instead replace the line by a curve such as
\(
    \{(x,y):y=f(x)\},
\)
or by several curves separating multiple spatial regions.  The consistency-set idea still applies, but the loss could now be measured in supremum or Hausdorff distance, and the rate should balance the local supply of crossing pairs against the smoothness and complexity of the boundary class.

\paragraph{Unknown or noisy connection rules.}
The exact interval reduction uses known hard radii.  Unknown radii could be estimated jointly with the boundary; more generally, one could use within- and across-community connection probabilities that decay with distance, as in random connection models \cite{meester1997random}.  Exact feasibility would then give way to a likelihood or approximate-consistency region, with a statistical divergence between the two connection laws replacing the annular scales above.  If the locations are also hidden, the boundary is identifiable only up to symmetries of the domain (for example, \(p\) and \(1-p\) are indistinguishable under horizontal reflection), and one faces joint geometric reconstruction and boundary estimation.

\medskip

We have deliberately kept the model to a vertical line, observed locations, and known hard thresholds in this note, as a first step towards more complicated boundary detection problems.  The simplicity makes the geometry pleasantly transparent while isolating a mechanism that survives in richer models.  The feasible-set construction provides a starting point for more general such problems on geometric networks.

\bibliographystyle{amsplain}
\bibliography{ref}

\begin{acks}
The initial ideas for this project, and many of the discussions that followed,
took shape over dinners at Himalayan Chimney, one of the best Indian
restaurants in Champaign--Urbana.
\end{acks}

\end{document}